\documentclass[oneside,english]{amsart}
\usepackage[T1]{fontenc}
\usepackage[latin9]{inputenc}
\usepackage{color}
\usepackage{float}
\usepackage{enumitem}
\usepackage{amstext}
\usepackage{amsthm}
\usepackage{amssymb}
\usepackage{graphicx}

\makeatletter

\providecommand{\tabularnewline}{\\}

\numberwithin{equation}{section}
\numberwithin{figure}{section}
\theoremstyle{plain}
\newtheorem{thm}{\protect\theoremname}
\theoremstyle{plain}
\newtheorem{prop}[thm]{\protect\propositionname}
\theoremstyle{plain}
\newtheorem{lem}[thm]{\protect\lemmaname}
\theoremstyle{definition}
\newtheorem{example}[thm]{\protect\examplename}
\theoremstyle{definition}
\newtheorem{defn}[thm]{\protect\definitionname}
\newlist{casenv}{enumerate}{4}
\setlist[casenv]{leftmargin=*,align=left,widest={iiii}}
\setlist[casenv,1]{label={{\itshape\ \casename} \arabic*.},ref=\arabic*}
\setlist[casenv,2]{label={{\itshape\ \casename} \roman*.},ref=\roman*}
\setlist[casenv,3]{label={{\itshape\ \casename\ \alph*.}},ref=\alph*}
\setlist[casenv,4]{label={{\itshape\ \casename} \arabic*.},ref=\arabic*}

\makeatother

\usepackage{babel}
\providecommand{\casename}{Case}
\providecommand{\definitionname}{Definition}
\providecommand{\examplename}{Example}
\providecommand{\lemmaname}{Lemma}
\providecommand{\propositionname}{Proposition}
\providecommand{\theoremname}{Theorem}

\begin{document}
\title{Number of moduli for an union of smooth curves in $\left(\mathbb{C}^{2},0\right)$.}
\author{Yohann Genzmer}
\begin{abstract}
In this article, we provide an algorithm to compute the number of
moduli of a germ of curve which is an union of germs of smooth curves
in the complex plane.
\end{abstract}

\maketitle

\section*{Introduction}

The problem of the determination of the number of moduli of a germ
of complex plane curve was addressed by Oscar Zariski in its famous
notes \cite{zariski}, where he focused on the case of a curve with
only one irreducible component. Since then, a lot of progress has
been made, and, up to our knowledge, the initial question can be considered
as mostly solved in the combination of the works of A. Hefez and M.
Hernandes \cite{MR2509045,MR2781209,MR2996882} and these of the author
\cite{YoyoBMS}. 

In this article, we proposed an approach of the general reducible
case in line with \cite{YoyoBMS}. From the study the module of vector
fields tangent to a curve $S$, which we refer to as the Saito module
of $S,$ we propose a formula for the generic dimension of the moduli
space of $S$ that can be easily implemented. The associated algorithm
is build upon the desingularization process of $S$, for which we
have already at our disposal, some classical and available routine
on many symbolic computation softwares. 

\section{Notation}

\global\long\def\un{\mathfrak{1}}%
\global\long\def\deux{\mathfrak{2}}%
Let $S$ be a germ of curve in the complex plane. According to \cite{zariskitop},
there exists a minimal process $E$ of desingularization that consists
in a sequence of elementary blowing-ups of points. We denote it by
\[
E=E_{1}\circ E_{2}\circ\cdots\circ E_{N}:\left(\widetilde{\mathbb{C}^{2}},D\right)\to\left(\mathbb{C}^{2},0\right).
\]
Here, $D=E^{-1}\left(0\right)$ is the \emph{exceptional divisor}
of $E.$ The strict transform of $S$ by any process of blowing-ups
$F$ will be refered to as $F^{\star}S$. The decomposition $D$ in
irreducible components is written
\[
D=\bigcup_{i=1}^{N}D_{i}
\]
where $D_{i}$ is the exceptional divisor of the elementary blowing-up
$E_{i}.$

Let $\left\{ t_{2},\cdots,t_{M}\right\} \subset D_{1}$ be the tangency
locus between $E_{1}^{\star}S$ and $D_{1}.$ For any $k=2,\ldots,M$,
$S_{k}$ stands for the germ of the curve $E_{1}^{\star}S$ at $t_{k}.$
Doing inductively the same construction for each curve $S_{k}$, we
finally obtain a family of curves 
\[
\left(S_{k}\right)_{k=2,\ldots,N},
\]
whose numbering is chosen so that $E_{k}$ is the blowing-up centered
at the tangency locus between $S_{k}$ and the exceptional divisor.
By extension, we set $S_{1}=S.$ 

Subsequently, the notation $\nu\left(\square\right)$ will stand for
the standard valuation of the object $\square:$
\begin{itemize}
\item if $S$ is a germ of curve, then $\nu\left(S\right)$ is the algebraic
multiplicity of any reduced local equation of $S$.
\item if $X$ is a germ of vector field written in some coordinates $X=a\partial_{x}+b\partial_{y},$
then 
\[
\nu\left(X\right)=\min\left(\nu\left(a\right),\nu\left(b\right)\right).
\]
\end{itemize}
If any confusion is possible, we will precise the point $p$ where
the valuation is evaluated. The associated notation will be $\nu_{p}\left(\square\right)$.

Since $E^{\star}S$ is smooth and transverse to the exceptional divisor,
one can consider for any component $D_{k}$, the number $n_{k}^{S}$
of components of $E^{\star}S$ attached to $D_{k}.$ 

We say that $D_{i}$ is \emph{in the neighborhood of $D_{k}$} if
$i\neq k$ and $D_{i}\cap D_{k}\neq\emptyset.$ The set of all indexes
$i\in\left\{ 1,\ldots,N\right\} $ such that $D_{i}$ is in the neighborhood
of $D_{k}$ will be denoted by $\mathfrak{N}\left(k\right).$ 

For $i\geq2,$ the component $D_{i}$ is the blowing-up of a point
which belongs to, either a single component $D_{j}$ or to a couple
of components $D_{j}$ and $D_{k}.$ The associated set of indexes
$\left\{ j\right\} $ or $\left\{ j,k\right\} $ is called the set
of \emph{parents }of $D_{i}$ and will be denoted by $\mathfrak{P}\left(i\right).$
By extension, we set $\mathcal{\mathfrak{P}}\left(1\right)=\emptyset.$
Notice that for any $i=2,\text{\ensuremath{\ldots}},M,$ one has 
\[
\mathcal{\mathfrak{P}}\left(i\right)=\left\{ 1\right\} .
\]

Finally, we introduce the following notation 
\[
\underset{n\quad}{\left\lfloor \begin{array}{c}
a\\
b
\end{array}\right.}=\left\{ \begin{array}{cl}
a & \textup{if \ensuremath{n\in\mathbb{N}} is even}\\
b & \textup{else}
\end{array}\right..
\]

\section{Proximity matrix of germ of curve $S\subset\left(\mathbb{C}^{2},0\right)$:
a combinatorial remark. }

Let $P$ be\emph{ the proximity matrix }of\emph{ $S$ }as defined
in \cite{MR2107253} or \cite{YoyoBMS}. The curve $S$ being given,
let $\Delta^{S}=\left(\Delta_{i}^{S}\right)_{i=1,\ldots,N}$ be any
element in $\left\{ 0,1\right\} ^{N}.$ Denote by $\delta_{k}^{S}$
the integer 
\[
\delta_{k}^{S}=\textup{card}\left\{ \left.i\in\mathfrak{P}\left(k\right)\right|\Delta_{k}^{S}=1\right\} .
\]
We consider the following vector of integers
\[
\mathfrak{S}^{S}=\left(\frac{\nu\left(S_{k}\right)-\delta_{k}^{S}}{2}+\underset{\nu\left(S_{k}\right)-\delta_{k}^{S}}{\left\lfloor \begin{array}{c}
\Delta_{k}^{S}\\
\frac{1}{2}
\end{array}\right.}\right)_{k=1,\ldots,N}
\]

We introduce the system of equations $\left(\mathcal{H}\right)$ whose
unknown variables are the vectors $\mathcal{E}^{S}=\left(\begin{array}{c}
\epsilon_{1}^{S}\\
\vdots\\
\epsilon_{N}^{S}
\end{array}\right)\in\mathbb{N}^{N}$ and $\Delta^{S}=\left(\Delta_{i}^{S}\right)_{i=1,\ldots,N}$ defined
by 
\[
\left(\mathcal{H}\right):~P^{-1}\mathcal{E}^{S}=\mathfrak{S}^{S}
\]
A solution $\left\{ \mathcal{E}^{S},\Delta^{S}\right\} $ is \emph{admissible}
if it satisfies the following \emph{compatibility conditions} :\emph{
}for any $k=1,\ldots,N$ 
\[
\left(\star\right):\left\{ \begin{array}{ccl}
\Delta_{k}=1 & \Longrightarrow & \epsilon_{k}^{S}\geq n_{k}^{S}\\
\Delta_{k}=0 & \Longrightarrow & \epsilon_{k}^{S}\geq2-\sum_{i\in\mathfrak{N}\left(k\right)}\Delta_{i}^{S}
\end{array}\right..
\]

One can choose an arbitrary vector $\Delta^{S}$, compute the associated
values of the integers $\delta_{k}^{S}$ and obtain the vector $\mathfrak{S}^{S}$.
Then, the invertible system $\left(\mathcal{H}\right)$ provides a
unique corresponding solution $\mathcal{E}^{S}.$ However, there is
no reason for this solution $\left\{ \mathcal{E}^{S},\Delta^{S}\right\} $to
satisfy the compatibility conditions $\left(\star\right)$. Nevertheless,
we will prove that 
\begin{prop}
\label{ExistenceAndUnicity}If $S$ is an union of germs of smooth
curves, there exists a unique choice of $\Delta^{S}$ such that the
associated solution of $\left(\mathcal{H}\right)$ satisfies the compatibility
conditions $\left(\star\right).$
\end{prop}

To prove the above proposition, first, let us establish a lemma that
describes the behaviour of the system $\left(\mathcal{H}\right)$
when one goes from $S$ to $S\cup l$ where $l$ is a generic smooth
curve. From now on, we suppose that $S$ is an union of germs of smooth
curves. 
\begin{lem}
\label{lemma2} Let $l$ be a generic germ of smooth curve. 
\begin{itemize}
\item If there exists $\Delta^{S}$ with $\Delta_{1}^{S}=0,$ such that
the solution of $\left(\mathcal{H}\right)$ is admissible, then the
same $\Delta^{S}$ provides an admissible solution for the curve $S\cup l$. 
\item If $\nu\left(S\right)$ is odd and there exists $\Delta^{S}$ with
$\Delta_{1}^{S}=1,$ such that the solution of $\left(\mathcal{H}\right)$
is admissible, then the same $\Delta^{S}$ provides an admissible
solution for the curve $S\cup l$.
\end{itemize}
\end{lem}

\begin{proof}
The lemma above can be seen on the behaviour of the system when $n_{1}$
is increased by $1.$ Since $S$ is an union of smooth curves, its
proximity matrix $P$ is written 
\[
P=\left(\begin{array}{ccccccc}
1 & -1 & \cdots & -1 & 0 & \cdots & 0\\
\vdots &  &  & \vdots &  &  & \vdots
\end{array}\right)
\]
with the number $-1$ repeated $M$ times on the first line. Therefore,
one can expand the expression of $\epsilon_{1}^{S}$ as below
\begin{equation}
\epsilon_{1}^{S}=\frac{\nu\left(S_{1}\right)-\delta_{1}^{S}}{2}+\underset{\nu\left(S_{1}\right)-\delta_{1}^{S}}{\left\lfloor \begin{array}{c}
\Delta_{1}^{S}\\
\frac{1}{2}
\end{array}\right.}-\sum_{k=2}^{M}\left(\frac{\nu\left(S_{k}\right)-\delta_{k}^{S}}{2}+\underset{\nu\left(S_{k}\right)-\delta_{k}^{S}}{\left\lfloor \begin{array}{c}
\Delta_{k}^{S}\\
\frac{1}{2}
\end{array}\right.}\right)\label{equation.fondamentale}
\end{equation}

By construction, $\delta_{1}^{S}=0.$ Now, if $\Delta_{1}^{S}=0,$
since the solution $\left\{ \mathcal{E}^{S},\Delta^{S}\right\} $
is admissible, one has $\epsilon_{1}^{S}\geq2-\sum_{i=2}^{M}\Delta_{i}^{S}.$
If $n_{1}$ is increased by one, it does not affect $\nu\left(S_{k}\right)$
for $k\geq2$ but it changes $\nu\left(S_{1}\right)$ into $\nu\left(S_{1}\right)+1.$
However, if $\nu\left(S_{1}\right)$ is even then
\[
\frac{\nu\left(S_{1}\cup l\right)}{2}+\underset{\nu\left(S_{1}\cup l\right)}{\left\lfloor \begin{array}{c}
0\\
\frac{1}{2}
\end{array}\right.}=\frac{\nu\left(S_{1}\right)}{2}+\underset{\nu\left(S_{1}\right)}{\left\lfloor \begin{array}{c}
0\\
\frac{1}{2}
\end{array}\right.}+1
\]
and if $\nu\left(S_{1}\right)$ is odd then 
\[
\frac{\nu\left(S_{1}\cup l\right)}{2}+\underset{\nu\left(S_{1}\cup l\right)}{\left\lfloor \begin{array}{c}
0\\
\frac{1}{2}
\end{array}\right.}=\frac{\nu\left(S_{1}\right)}{2}+\underset{\nu\left(S_{1}\right)}{\left\lfloor \begin{array}{c}
0\\
\frac{1}{2}
\end{array}\right.}
\]
Thus setting $\epsilon_{1}^{S\cup l}=\epsilon_{1}^{S}$ or $\epsilon_{1}^{S}+1$
depending on $\nu\left(S_{1}\right)$ being odd or even and 
\[
\begin{cases}
\Delta_{i}^{S\cup l}=\Delta_{i}^{S} & i=1,\ldots,N\\
\epsilon_{i}^{S\cup l}=\epsilon_{i}^{S} & i\neq2
\end{cases}
\]
 yields an admissible solution of the system $\left(\mathcal{H}\right)$
for $S\cup l$. 

Now, if $\nu\left(S_{1}\right)$ is odd and $\Delta_{1}^{S}=1$, then
one has 
\[
\epsilon_{1}^{S}=\frac{\nu\left(S_{1}\right)+1}{2}-\sum_{k=2}^{M}\left(\frac{\nu\left(S_{k}\right)-\delta_{k}^{S}}{2}+\underset{\nu\left(S_{k}\right)-\delta_{k}^{S}}{\left\lfloor \begin{array}{c}
\Delta_{k}^{S}\\
\frac{1}{2}
\end{array}\right.}\right).
\]
and $\epsilon_{1}^{S}\geq n_{1}^{S}$. The multiplicity $\nu\left(S_{1}\right)$
being odd, one has 
\[
\frac{\nu\left(S_{1}\cup l\right)}{2}+\underset{\nu\left(S_{1}\cup l\right)}{\left\lfloor \begin{array}{c}
1\\
\frac{1}{2}
\end{array}\right.}=\frac{\nu\left(S_{1}\right)+1}{2}+1.
\]
Thus setting $\epsilon_{1}^{S\cup l}=\epsilon_{1}^{S}+1$ yields an
admissible solution of the system for $S\cup l$ since 
\[
\epsilon_{1}^{S\cup l}=\epsilon_{1}^{S}+1\geq n_{1}^{S}+1=n_{1}^{S\cup l}.
\]
\end{proof}
\begin{proof}[Proof of Proposition \ref{ExistenceAndUnicity}.]
The proof is an induction on the length of the desingularization
of $S.$ First, let us prove the proposition for a curve $S$ desingularized
after one blowing-up. The system $\left(\mathcal{H}\right)$ reduces
to the sole equation 
\begin{equation}
{\color{red}{\color{black}\epsilon_{1}^{S}=\frac{\nu\left(S_{1}\right)}{2}+\underset{\nu\left(S_{1}\right)}{\left\lfloor \begin{array}{c}
\Delta_{1}^{S}\\
\frac{1}{2}
\end{array}\right.}}}\label{eq:1}
\end{equation}
If $\nu\left(S_{1}\right)=n_{1}^{S}=1$ or $2$ then $\Delta_{1}^{S}=1$
is the unique admissible choice since $\epsilon_{1}^{S}$ is respectively
equal to $1\textup{ and }2$, which are all bigger than the respective
$n_{1}^{S},$ whereas if $\Delta_{1}^{S}=0$ one finds always $1$
which is not bigger than $2=2-\sum_{i\in\mathfrak{N}\left(1\right)}\Delta_{i}^{S}.$
If $\nu\left(S_{_{1}}\right)=n_{1}^{S}\geq3$ then $\epsilon_{1}^{S}<n_{1}^{S}$
thus $\Delta_{1}^{S}=1$ is excluded. However, $\Delta_{1}^{S}=0$
brings an admissible solution to the equation (\ref{eq:1}) since
$\epsilon_{1}^{S}\geq2.$ 

For the inductive step, let us consider the unique choice $\left(\Delta_{k}^{S,0}\right)_{k=2,\cdots,M}$
provided by the inductive application of the proposition to each $S_{i}.$
In the same way, consider the unique choice $\left(\Delta_{k}^{S,1}\right)_{k=2,\cdots,M}$
obtained when the proposition is applied to each $S_{i}\cup D_{1}$.
From these data, we can compute the value of $\epsilon_{1}^{S,\star}$,$\ \star=0,1$
provided that we choose $\Delta_{1}^{S}=0$ or $1.$ If we choose
$\Delta_{1}^{S}=1$ then it can be seen that in this situation, $\delta_{1}^{S,1}=0,$
$\delta_{k}^{S,1}=1$ for $k\geq2$. In the same way, if we set $\Delta_{1}^{S}=0$
then $\delta_{k}^{S,0}=0$ for $k\geq2.$ Therefore, following (\ref{equation.fondamentale})
one has 
\begin{align*}
\epsilon_{1}^{S,1}+\epsilon_{1}^{S,0} & =\nu\left(S_{1}\right)+\underbrace{\underset{\nu\left(S_{1}\right)}{\left\lfloor \begin{array}{c}
1\\
\frac{1}{2}
\end{array}\right.}+\underset{\nu\left(S_{1}\right)}{\left\lfloor \begin{array}{c}
0\\
\frac{1}{2}
\end{array}\right.}}_{=1}\\
 & -\sum_{i=2}^{M}\left(\frac{\nu\left(S_{i}\right)-1}{2}+\underset{\nu\left(S_{i}\right)-1}{\left\lfloor \begin{array}{c}
\Delta_{i}^{S,1}\\
\frac{1}{2}
\end{array}\right.}\right)-\sum_{i=2}^{M}\left(\frac{\nu\left(S_{i}\right)}{2}+\underset{\nu\left(S_{i}\right)}{\left\lfloor \begin{array}{c}
\Delta_{i}^{S,0}\\
\frac{1}{2}
\end{array}\right.}\right).
\end{align*}
Observe that $\nu\left(S_{1}\right)-\sum_{i=2}^{M}\nu\left(S_{i}\right)=n_{1}^{S}.$
Thus, the relation above reduces to 
\begin{align*}
\epsilon_{1}^{S,1}+\epsilon_{1}^{S,0} & =n_{1}^{S}+1-\sum_{i=2}^{M}\frac{-1}{2}\underset{\nu\left(S_{i}\right)-1}{+\left\lfloor \begin{array}{c}
\Delta_{i}^{S,1}\\
\frac{1}{2}
\end{array}\right.}+\underset{\nu\left(S_{i}\right)}{\left\lfloor \begin{array}{c}
\Delta_{i}^{S,0}\\
\frac{1}{2}
\end{array}\right.}\\
 & =n_{1}^{S}+1-\sum_{i=2}^{M}\underset{\nu\left(S_{i}\right)}{\left\lfloor \begin{array}{c}
\Delta_{i}^{S,0}\\
\Delta_{i}^{S,1}
\end{array}\right..}
\end{align*}
Now, if $\nu\left(S_{i}\right)$ is even then $\underset{\nu\left(S_{i}\right)}{\left\lfloor \begin{array}{c}
\Delta_{i}^{S,0}\\
\Delta_{i}^{S,1}
\end{array}\right.}=\Delta_{i}^{S,0}.$ According to Lemma \ref{lemma2}, if $\Delta_{i}^{S,0}=0$ then
$\Delta_{i}^{S,1}=0.$ Moreover, if $\nu\left(S_{i}\right)$ is odd,
if $\Delta_{i}^{S,0}=1$ then $\Delta_{i}^{S,1}=1$ thus in any case,
$\underset{\nu\left(S_{i}\right)}{\left\lfloor \begin{array}{c}
\Delta_{i}^{S,0}\\
\Delta_{i}^{S,1}
\end{array}\right.}=\Delta_{i}^{S,0}.$ Finally, we are led to the relation 
\begin{equation}
\epsilon_{1}^{S,1}+\epsilon_{1}^{S,0}=n_{1}^{S}+1-\sum_{i=2}^{M}\Delta_{i}^{S,0}.\label{egaliteSuper}
\end{equation}
Hence, one of the following inequalities holds 
\[
\epsilon_{1}^{S,1}\geq n_{1}^{S}
\]
or
\[
\epsilon_{1}^{S,0}\geq2-\sum_{i=2}^{M}\Delta_{i}^{S,0}
\]
the two being mutually exclusive according to (\ref{egaliteSuper}).
By induction, this concludes the proof. 
\end{proof}
\begin{example}
Suppose that $S$ is reduced by two successive blowing-ups, then its
proximity matrix is written 
\[
\left(\begin{array}{cc}
1 & -1\\
0 & 1
\end{array}\right).
\]
Below, we present the unique choice of $\Delta^{S}=\left(\star,\star\right)\in\left\{ 0,1\right\} ^{2}$
depending on $n_{1}^{S}$ and $n_{2}^{S}$ that leads to an admissible
solution of $\left(\mathcal{H}\right)$.
\begin{figure}[H]
\begin{centering}
\includegraphics[scale=0.25]{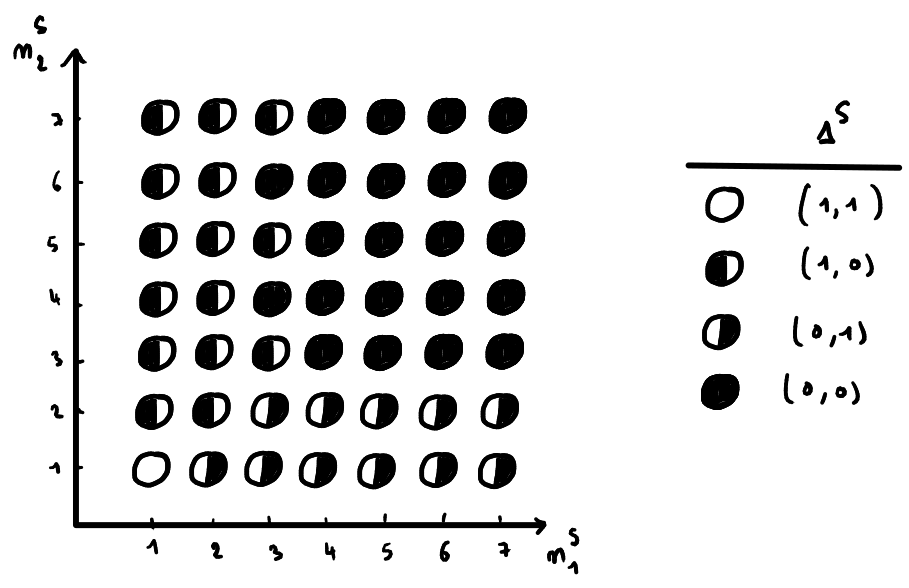}
\par\end{centering}
\caption{Unique admissible choice of $\Delta^{S}=\left(\Delta_{1}^{S},\Delta_{2}^{S}\right).$}
\end{figure}
 For instance, if $n_{1}^{S}=3$ and $n_{2}^{S}=5$ then $\nu\left(S_{1}\right)=8$
and $\nu\left(S_{2}\right)=5$. Setting $\Delta^{S}=\left(1,0\right)$
yields 
\[
\mathcal{E}^{S}=\left(\begin{array}{c}
3\\
2
\end{array}\right).
\]
One can check that 
\[
P^{-1}\mathcal{E}^{S}=\left(\begin{array}{cc}
1 & 1\\
0 & 1
\end{array}\right)\left(\begin{array}{c}
3\\
2
\end{array}\right)=\left(\begin{array}{c}
5\\
2
\end{array}\right)=\left(\begin{array}{c}
\frac{8-0}{2}+1\\
\frac{5-1}{2}
\end{array}\right)=\mathfrak{O}^{S}
\]
and 
\[
\epsilon_{1}^{S}=3\geq n_{1}^{S}=3\qquad\epsilon_{2}^{S}=2\geq2-0
\]
so the solution $\left\{ \mathcal{E}^{S},\Delta^{S}\right\} $ is
admissible. 
\end{example}

\section{Saito vector field.}

In this section, $S$ is any germ of curves - not necessarly an union
of smooth curves.

\subsection{Definition of a Saito vector field for a curve.}

Let $\textup{Der}\left(S\right)$ be the set of germ of vector fields
$X$ tangent to $S,$ i.e., such that for a reduced equation $f$
of $S,$ one has 
\[
X\cdot f\in\left(f\right).
\]

According to \cite{MR586450}, $\textup{Der}\left(S\right)$ is a
free $\mathcal{O}_{2}-$module of rank $2$ and any basis $\left\{ X_{1},X_{2}\right\} $
of $\textup{Der}\left(S\right)$ will be said a \emph{Saito basis}
for $S.$ The \emph{number of Saito} of $S$ is 
\[
\mathfrak{s}\left(S\right)=\min_{X\in\textup{Der}\left(S\right)}\nu\left(X\right)=\min\left(\nu\left(X_{1}\right),\nu\left(X_{2}\right)\right).
\]

A vector field $X\in\textup{Der}\left(S\right)$ is said to be \emph{optimal}
for $S$ if 
\[
\nu\left(X\right)=\mathfrak{s}\left(S\right).
\]

If $E$ is any process of blowing-up, we denote by $X^{E}$ the \emph{divided
pull-back} vector field of $X$ by $E$. It is a familly a vector
field parametrized by the point of the exceptionnal divisor : for
any $c\in E^{-1}\left(0\right),$ $\left(X^{E}\right)_{c}$ is written
$\frac{Y}{u^{a}}$ (or $\frac{Y}{u^{a}v^{b}}$) where $Y$ projects
onto $X$ with respect to $E$ and $u^{a}$ is the maximal power of
$u$ that divides $Y,$ where $u$ ( or $uv=0$ ) is a local equation
of $E^{-1}\left(0\right)$ at $c.$ An alternative way to construct
$X^{E}$ is the following : the vector field $X$ induces a saturated
foliation $\mathcal{F}$ at the origin of $\mathbb{C}^{2}.$ The foliation
$\mathcal{F}$ can be pulled-back by $E$ in $E^{\star}\mathcal{F}$
which defines a saturated foliation in the neighborhood of $D$. The
vector field $\left(X^{E}\right)_{c}$ is any generator of the latter
at $c.$ 

The vector field $X$ is said to be \emph{dicritical} if $X^{E_{1}}$
is generically transversal to the exceptional divisor $E_{1}^{-1}\left(0\right).$

Below, we recall some material established in \cite{genzmer2020saito}.
\begin{thm}
\label{SaitoBasisForm}Let $S$ be a curve \emph{generic} in its moduli
space. Then there exists a Saito basis $\left\{ X_{1},X_{2}\right\} $
for $S$ with one of the following forms
\begin{itemize}
\item if $\nu\left(S\right)$ is even
\begin{itemize}
\item[$\left(\mathfrak{E}\right)$] : $\nu\left(X_{1}\right)=\nu\left(X_{2}\right)=\frac{\nu\left(S\right)}{2}$
and $X_{1}$ and $X_{2}$ are non dicritical.
\item[$\left(\mathfrak{E}_{d}\right)$] : $\nu\left(X_{1}\right)=\nu\left(X_{2}\right)-1=\frac{\nu\left(S\right)}{2}-1$
and $X_{1}$ and $X_{2}$ are dicritical.
\end{itemize}
\item if $\nu\left(S\right)$ is odd
\begin{itemize}
\item[$\left(\mathfrak{O}\right)$] : $\nu\left(X_{1}\right)=\nu\left(X_{2}\right)-1=\frac{\nu\left(S\right)-1}{2}$
and $X_{1}$ and $X_{2}$ are non dicritical.
\item[$\left(\mathfrak{O}_{d}\right)$] : $\nu\left(X_{1}\right)=\nu\left(X_{2}\right)=\frac{\nu\left(S\right)-1}{2}$
and $X_{1}$ and $X_{2}$ are dicritical.
\end{itemize}
\end{itemize}
In particular, the Saito number of $S$ is equal to 
\[
\mathfrak{s}\left(S\right)=\frac{\nu\left(S\right)}{2}-\underset{\nu\left(S\right)\qquad\quad}{\left\lfloor \begin{array}{c}
1-\Delta\\
\frac{1}{2}
\end{array}\right.}
\]
where $\Delta=\left\{ \begin{array}{cl}
1 & \textup{ if \ensuremath{S} is of type \ensuremath{\left(\mathfrak{O}\right)} or \ensuremath{\left(\mathfrak{E}\right)}}\\
0 & \textup{else}
\end{array}.\right.$

If $S$ is of type $\left(\mathfrak{E}_{d}\right)$ or $\left(\mathfrak{O}_{d}\right)$
and has \emph{no free point} - see \cite{genzmer2020saito} - then
there exists a basis of the following form
\begin{itemize}
\item[$\left(\mathfrak{E}_{d}^{\prime}\right)$] : $\nu\left(X_{1}\right)=\nu\left(X_{2}\right)-2=\frac{\nu\left(S\right)}{2}-1$
and $X_{1}$ is dicritical but not $X_{2}.$ 
\item[$\left(\mathfrak{O}_{d}^{\prime}\right)$] : $\nu\left(X_{1}\right)=\nu\left(X_{2}\right)-1=\frac{\nu\left(S\right)-1}{2}$
and $X_{1}$ is dicritical but not $X_{2}.$ 
\end{itemize}
\end{thm}

A basis given by the above result is said to \emph{be adapted. }An
adapted basis behaves well with respect to the blowing-up : indeed,
in any case, if $\left\{ X_{1},X_{2}\right\} $ is an adapted basis
for $S$ then for any $c\in D_{1},$ the family 
\[
\left\{ X_{1}^{E_{1}},X_{2}^{E_{1}}\right\} 
\]
is a Saito basis for $\left(E_{1}^{\star}S\right)_{c}$ or $\left(E_{1}^{\star}S\cup D_{1}\right)_{c}$
depending on the type of the basis. Notice that this property does
not hold for any Saito basis and that the basis above may not be adapted. 

For the sake of simplicity, we will say that $S$ is of class $\mathfrak{1}$
if $S$ is of type $\left(\mathfrak{E}\right)$ or $\left(\mathfrak{O}_{d}\right).$
Otherwise, we will say that $S$ is of class $\mathfrak{2}.$ The
main difference between the two classes is that the vector fields
of an adapted basis for a curve of type $\mathfrak{1}$ share the
same valuations, whereas they are different for a curve of class $\mathfrak{2}.$
In particular, for a curve of class $\mathfrak{1}$, any Saito basis
is actually adapted. 

To keep track of the type ($\mathfrak{1}$ or $\mathfrak{2}$) of
the successive blowing-ups of the curve $S,$ we introduce the notion
of \emph{relative strict transform of $S.$}
\begin{defn}
The \emph{relative strict transform} of $S$ by $E$, denoted by $S^{E}$,
is the following union of curves 
\[
S^{E}=E^{\star}S\cup\bigcup_{i\in J\subset\left\{ 1,\cdots,N\right\} }D_{i}
\]
 where $J$ is inductively defined as follows : 
\begin{center}
$i\in J$ $\Longleftrightarrow$ $S_{i}\cup\text{\ensuremath{\bigcup}}_{j\in\mathfrak{P}\left(i\right)\cap J}D_{j}$
is of type $\left(\mathfrak{E}\right)$ or $\left(\mathfrak{O}\right).$
\par\end{center}

\end{defn}

Finally, we are able to introduce the main object of interest here. 
\begin{defn}
A germ of vector field is said to be \emph{Saito} for $S$ if for
any intermediate process of blowing-ups $E^{1,k}$ of $E$ the vector
field $X^{E^{1,k}}$ is optimal for $S^{E^{1,k}}$.

In other words, a vector field is said Saito for $S$ if it is optimal
for $S$ and if this property propagates all along the process of
desingularization of $S.$ 

The definition be given, there is apparently no reason for such a
vector field to exist in general. However, we will see that this is
actually the case in particular, for unions of germs of smooth curves. 
\end{defn}

\begin{example}
\label{exa:Let-us-consider}Let us consider the curve $S$ defined
by 
\[
S=\left\{ x\left(x+y^{2}\right)y\left(y+x^{2}\right)=0\right\} .
\]

It can be checked that $S$ is of type $\left(\mathfrak{E}\right)$
and that 
\begin{align*}
X_{1} & =\left(xy^{2}+x^{2}\right)\partial_{x}+\left(2y^{3}+2xy\right)\partial_{y}\\
X_{2} & =\left(2x^{3}+2xy\right)\partial_{x}+\left(x^{2}y+y^{2}\right)\partial_{y}
\end{align*}
is an adapted basis. In particular, the Saito number of $S$ is 
\[
\mathfrak{s}\left(S\right)=2.
\]
Moreover, after one blowing-up, $X_{1}+X_{2}$ is given in the chart
containing the singular point $\left(0,0\right)$ of $S_{2}$ by 
\[
\left(X_{1}+X_{2}\right)^{E_{1}}=x\left(xy^{2}+2x+2y+1\right)\partial_{x}+y\left(xy^{2}-x-y+1\right)\partial_{y}
\]
which is of multiplicity of $1$ and tangent to the radial vector
field at order $1$. In the other chart, the same occurs for the singular
point of $S_{3}$. Therefore, $X_{1}+X_{2}$ is Saito for $S.$ Notice
that, $\left(X_{1}+X_{2}\right)^{E_{1}}$ admits an other singular
point whose coordinates are $\left(0,1\right)$ in the coordinates
of the chart above. At $\left(0,1\right)$, its linear part is not
trivial and has two non vanishing eigenvalues whose quotient is not
a non negative rational number. In particular, according to \cite{diff},
$X_{1}+X_{2}$ admits a smooth invariant curve that is neither contained
in $S$ nor tangent to a component of $S.$ 

\begin{figure}
\begin{centering}
\includegraphics[scale=0.3]{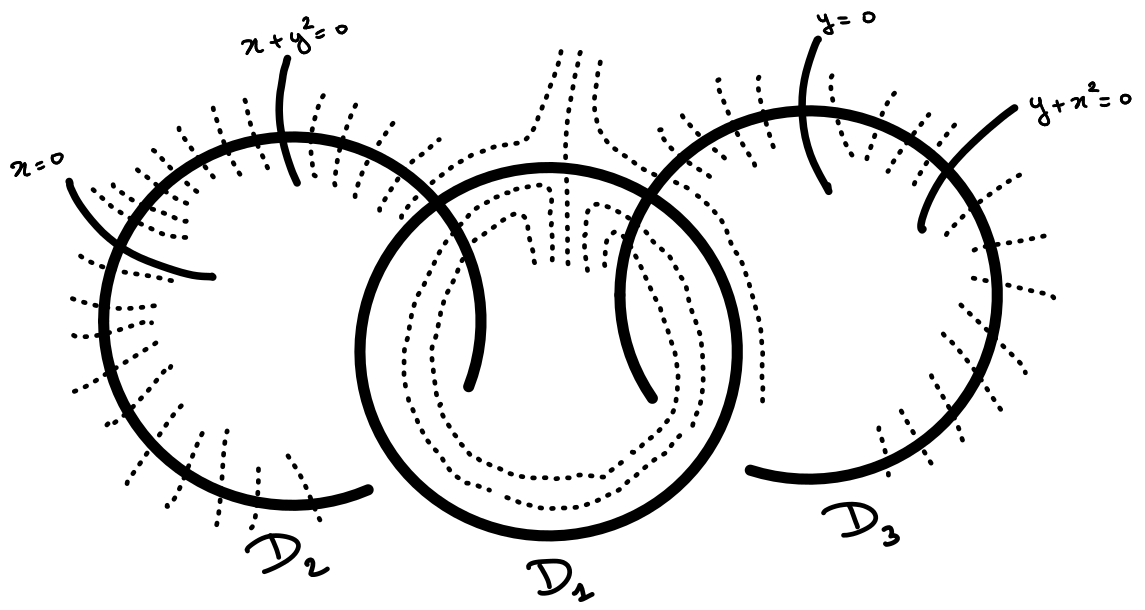}
\par\end{centering}
\caption{\label{fig:Topology-of-the}Topology of the leaves of the vector field
$\left(X_{1}+X_{2}\right)^{E}$}

\end{figure}
\end{example}

\subsection{Numerical property of a Saito vector field.}

Let us investigate the relation between the topological data associated
to a Saito vector field and the combinatorial property of the proxmity
matrix established before. 

First, let us recall some results from \cite{hertlingfor}.

Let $\mathfrak{M}$ be the sheaf generated by the global functions
$h\circ E$ with $h\in\mathcal{O}_{2}$ and $h\left(0\right)=0$.
It is a simple matter to get the following decomposition of sheaves
\[
\mathfrak{M}=\mathcal{O}\left(-\sum_{i=1}^{N}\rho_{i}^{E}D_{i}\right)
\]
where the integers $\rho_{i}^{E}$ are known as \emph{the multiplicities
of $D$}. The number $\rho_{i}^{E}$ is also the multiplicity of a
curve whose strict transform by $E$ is smooth and attached to a regular
point of $D_{i}$.

The integer $\textup{val}_{X}\left(D_{i}\right)$ refers to \emph{the
non-dicritical valence} of $D_{i}$, which is the number of non $X^{E}-$invariant
components of $D$ that are in the neighborhood of $D_{i}$. 

The following definitions are proposed in \cite{hertlingfor}.
\begin{defn}
Let $X$ be a germ of vector field given by 
\[
\omega=a\left(x,y\right)\partial_{x}+b\left(x,y\right)\partial_{y}
\]
\begin{enumerate}
\item Let $\left(S,p\right)$ be a germ of smooth invariant curve. If, in
some coordinates, $S$ is the curve $\left\{ x=0\right\} $ and $p$
the point $(0,0)$, then the integer $\nu\left(b\left(0,y\right)\right)$
is called \emph{the indice} of $X$ at $p$ with respect to $S$ and
is denoted by 
\[
\textup{Ind}\left(X,S,p\right).
\]
\item Let $\left(S,p\right)$ be a germ of smooth non-invariant curve. If,
in some coordinates, $S$ is the curve $\left\{ x=0\right\} $ and
$p$ the point $(0,0)$, then the integer $\nu\left(a\left(0,y\right)\right)$
is called \emph{the tangency order} of $X$ with respect to $S$ and
is denoted 
\[
\textup{Tan}\left(X,S,p\right).
\]
 
\end{enumerate}
\end{defn}

The equality below is proved in \cite{hertlingfor} and specializes
to a result of \cite{camacho} if $X^{E}$ leaves invariant $D.$ 
\begin{prop}
\label{prop.Hertling}The multiplicity of $X$ satisfies the equality
\[
\nu\left(X\right)+1=\sum_{i=1}^{N}\rho_{i}^{E}\epsilon_{i}\left(X,E\right)
\]
where 
\begin{enumerate}
\item if $D_{i}$ is non invariant by $X^{E}$, $\epsilon_{i}\left(X,E\right)=-\textup{val}_{X}\left(D_{i}\right)+\sum_{c\in D_{i}}\textup{Ind}\left(X^{E},D_{i},c\right)$. 
\item if $D_{i}$ is invariant by $X^{E}$, $\epsilon_{i}\left(X,E\right)=2-\textup{val}_{X}\left(D_{i}\right)+\sum_{c\in D_{i}}\textup{Tan}\left(X^{E},D_{i},c\right)$. 
\end{enumerate}
\end{prop}

\begin{thm}
\label{Saito.est.admiss}If $X$ is a Saito vector field then setting
\[
\text{\ensuremath{\mathcal{E}^{S}=\left(\epsilon_{i}\left(X,E\right)\right)_{i=1,\ldots,N}} }
\]
and $\Delta^{S}=\left(\Delta_{i}^{S}\right)_{i=1,\ldots,N}$ such
that 
\[
\Delta_{i}^{S}=\left\{ \begin{array}{cl}
1 & \textup{ if \ensuremath{D_{i}} is invariant by \ensuremath{X^{E}}}\\
0 & \textup{else}
\end{array}\right.
\]
 yields an admissible solution $\left\{ \mathcal{E}^{S},\Delta^{S}\right\} $
of $\left(\mathcal{H}\right).$
\end{thm}

\begin{proof}
For $k=1,\ldots,N,$ let $E^{\prime}$ be the intermediate process
of blowing-ups that leads to $S_{k}$ and $E^{k}$ such that 
\[
E=E^{\prime}\circ E^{k}.
\]
Let us denote by $p$ the point of attachement of $S_{k}$ to the
exceptional divisor of $E^{\prime}.$ The vector field $X^{E^{\prime}}$
being optimal for $\left(S^{E^{\prime}}\right)_{p}$, we have 
\begin{align}
\nu_{p}\left(X^{E^{\prime}}\right)+1 & =\sum_{i=k}^{N}\rho_{i}^{E^{k}}\epsilon_{i}\left(X^{E^{\prime}},E^{k}\right)=\mathfrak{s}\left(\left(S^{E^{\prime}}\right)_{p}\right)+1\label{equation.boubou}\\
 & =\frac{\nu\left(\left(S^{E^{\prime}}\right)_{p}\right)}{2}+\underset{\nu\left(\left(S^{E^{\prime}}\right)_{p}\right)\quad}{\left\lfloor \begin{array}{c}
\Delta_{k}\\
\frac{1}{2}
\end{array}\right.}.\nonumber 
\end{align}
Now, it can be seen that for $i\neq k$
\[
\epsilon_{i}\left(X^{E^{\prime}},E^{k}\right)=\epsilon_{i}\left(X,E\right)
\]
and that 
\[
\epsilon_{k}\left(X^{E^{\prime}},E^{k}\right)=\epsilon_{k}\left(X,E\right)+\delta_{k}^{S}.
\]
Since $\rho_{k}^{E^{k}}=1$ and $\nu\left(\left(S^{E^{\prime}}\right)_{p}\right)=\nu\left(S_{k}\right)+\delta_{k}^{S}$,
the relation (\ref{equation.boubou}) is written
\[
\sum_{i=k}^{N}\rho_{i}^{E^{k}}\epsilon_{i}\left(X,E\right)=\frac{\nu\left(S_{k}\right)-\delta_{k}^{S}}{2}+\underset{\nu\left(S_{k}\right)-\delta_{k}^{S}\quad}{\left\lfloor \begin{array}{c}
\Delta_{k}\\
\frac{1}{2}
\end{array}\right.}.
\]
Now following \cite{MR2107253}, the matrix defined by 
\[
\left(\rho_{i}^{E^{k}}\right)_{N\geq i\geq k\geq1}
\]
is an upper triangular invertible matrix and its inverse is the proximity
matrix $P$. Thus, the vectors $\mathcal{E}^{S}$ and $\Delta^{S}$
as defined in the statement provide a solution to the system $\left(\mathcal{H}\right).$
Moreover, if $\Delta_{k}=0$, then
\begin{align*}
\epsilon_{k}\left(X,E\right)-2+\sum_{i\in\mathfrak{N}\left(k\right)}\Delta_{i}^{S} & =\epsilon_{k}\left(X,E\right)-2+\textup{val}_{X}\left(D_{k}\right)\\
 & =\sum_{c\in D_{k}}\textup{Tan}\left(X^{E},D_{k},c\right)\geq0
\end{align*}
and if $\Delta_{k}=1$ then 
\begin{align*}
\epsilon_{k}\left(X,E\right) & =-\textup{val}_{X}\left(D_{k}\right)+\sum_{c\in D_{k}}\textup{Ind}\left(X^{E},D_{k},c\right)\\
 & =-\sum_{i\in\mathfrak{N}\left(k\right)}\Delta_{i}^{S}+\sum_{c=D_{k}\cap D_{i},\ i\in\mathfrak{N}\left(k\right)}\textup{Ind}\left(X^{E},D_{k},c\right)\\
 & \quad+\sum_{c\neq D_{k}\cap D_{i},\ i\in\mathfrak{N}\left(k\right)}\textup{Ind}\left(X^{E},D_{k},c\right)
\end{align*}
It can be seen that if $\Delta_{i}^{S}=1$ then $\textup{Ind}\left(X^{E},D_{k},D_{k}\cap D_{i}\right)\geq1.$
Moreover, for any regular component of $S^{E}$ attached to $D_{k}$
at $c$, one has
\[
\textup{Ind}\left(X^{E},D_{k},c\right)\geq1.
\]
Thus, 
\[
\sum_{c\neq D_{k}\cap D_{i},\ i\in\mathfrak{N}\left(k\right)}\textup{Ind}\left(X^{E},D_{k},c\right)\geq n_{k}^{S}.
\]
Finally, we are led to 
\[
\epsilon_{k}\left(X,E\right)\geq n_{k}^{S}.
\]
Therefore, the solution $\left\{ \mathcal{E}^{S},\Delta^{S}\right\} $
is admissible. 
\end{proof}
\begin{example}
The proximity matrix of Example (\ref{exa:Let-us-consider}) is 
\[
P=\left(\begin{array}{ccc}
1 & -1 & -1\\
0 & 1 & 0\\
0 & 0 & 1
\end{array}\right)
\]
and one has 
\[
\nu\left(S_{1}\right)=4,\ \nu\left(S_{2}\right)=2,\nu\left(S_{3}\right)=2,\ n_{1}^{S}=0,\ n_{2}^{S}=2,\ n_{3}^{S}=2.
\]
Picture (\ref{fig:Topology-of-the}) ensures that 
\[
\epsilon_{1}^{S}=1,~\epsilon_{2}^{S}=\epsilon_{3}^{S}=1\quad\Delta_{1}^{S}=1,\ \Delta_{2}^{S}=\Delta_{3}^{S}=0.
\]
Thus, we obtain that $\delta_{1}^{S}=0,\ \delta_{2}^{S}=\delta_{3}^{S}=1.$
Finally, one can check that 
\[
P^{-1}\mathcal{E}^{S}=\left(\begin{array}{ccc}
1 & 1 & 1\\
0 & 1 & 0\\
0 & 0 & 1
\end{array}\right)\left(\begin{array}{c}
1\\
1\\
1
\end{array}\right)=\left(\begin{array}{c}
3\\
1\\
1
\end{array}\right)=\left(\begin{array}{c}
\frac{4-0}{2}+1\\
\frac{2-1}{2}+\frac{1}{2}\\
\frac{2-1}{2}+\frac{1}{2}
\end{array}\right)=\mathfrak{S}^{S}
\]
\end{example}

\subsection{Existence of a Saito vector field. }

Below, we establish the existence of a Saito vector field for an union
of germs of smooth curves. 
\begin{thm}
\label{thm.12}Let $S$ be a generic curve whose components are all
smooth. Then
\begin{enumerate}
\item there exists a vector field $X$ Saito for $S.$ 
\item there exists $l$ a germ of smooth curve such that $S\cup l$ has
no Saito basis of type $\left(\mathfrak{E}_{d}^{\prime}\right).$
\end{enumerate}
\end{thm}

\begin{proof}
The proof is an induction on the length of the desingularization of
$S.$ 

If the length is zero, then $S$ is a smooth curve. In some coordinates
$\left(x,y\right)$ such that $S=\left\{ x=0\right\} $, the vector
field $X=\partial_{x}$ is Saito for $S.$ Moreover, if $l$ is the
line $\left\{ y=0\right\} ,$ then the family $\left\{ x\partial_{x},y\partial_{y}\right\} $
is an adapted basis for $S$ which is not of type $\left(\mathfrak{E}_{d}^{\prime}\right).$

We suppose now that the length of the desingularization of $S$ is
strictly positive. 
\begin{center}
\begin{figure}[H]
\begin{centering}
\includegraphics[scale=0.3]{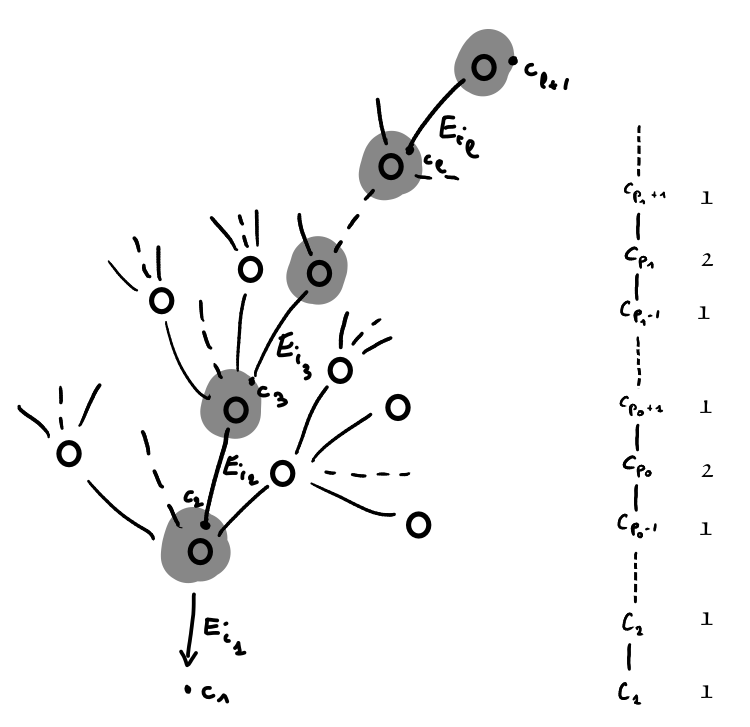}
\par\end{centering}
\caption{\label{fig:A-branch-of}A branch of the tree of desingularization
of $S$ and the type of its successive blowing-ups.}
\end{figure}
\par\end{center}

We consider a \emph{branch }of the tree of desingularization, that
is a sequence of blowing-ups $\left\{ E_{i_{k}}\right\} _{k=1\cdots j}$
where $i_{1}=1$ and $E_{i_{k}}$ is centered at a point $c_{k}$
which belongs to exceptional divisor of $E_{i_{k-1}}.$ We will denote
by $E^{k,l}$ the composition 
\[
E^{k,l}=E_{i_{k}}\circ E_{i_{k+1}}\circ\cdots\circ E_{i_{l}}.
\]
Notice that $c_{1}$ is the origin of $\mathbb{C}^{2}.$ We denote
by $c_{j+1}$ any point of attachement of the strict transform of
$S$ along the exceptional divisor of $E_{i_{j}}.$

Let us consider $\left\{ X_{1},X_{2}\right\} $ an adapted basis for
$S$. We can suppose that the classes of the successive curves along
the branch have the form in Figure \ref{fig:A-branch-of}. Notice
that in this picture, the index $p_{0}$ may be equal to $1$, so
that, the first curve $S_{1}$ is of class $\mathfrak{2}$. Hence,
the description of the branch covers actually the general case.

Any family $\left\{ X_{1}^{E^{1,k}},X_{2}^{E^{1,k}}\right\} $ for
$k\leq p_{0}-1$ is a Saito basis for $\left(S^{E^{1,k}}\right)_{c_{k+1}}$:
indeed, the curve $\left(S^{E^{1,k}}\right)_{c_{k+1}}$ being of class
$\mathfrak{1}$ for $k\leq p_{0}-2$, the Saito basis $\left\{ X_{1}^{E^{1,k}},X_{2}^{E^{1,k}}\right\} $
is also adapted. Taking if necessary a generic combination of $X_{1}$
and $X_{2}$, we can suppose that in the basis $\left\{ X_{1}^{E^{1,p_{0}-1}},X_{2}^{E^{1,p_{0}-1}}\right\} $,
one has 
\begin{equation}
\nu_{c_{p_{0}}}\left(X_{1}^{E^{1,p_{0}-1}}\right)\leq\nu_{c_{p_{0}}}\left(X_{2}^{E^{1,p_{0}-1}}\right),\label{ineg.1}
\end{equation}
and $X_{1}^{E^{1,p_{0}-1}}$ is optimal for $\left(S^{E^{1,p_{0}-1}}\right)_{c_{p_{0}}}.$
Since, the latter is a curve of class $\mathfrak{2},$ there exists
$c\in\mathbb{C}$ such that 
\[
\nu_{c_{p_{0}}}\left(X_{1}^{E^{1,p_{0}-1}}\right)<\nu_{c_{p_{0}}}\left(\underbrace{X_{2}^{E^{1,p_{0}-1}}+cX_{1}^{E^{1,p_{0}-1}}}_{\tilde{X}_{2}}\right)
\]
making of the basis $\left\{ X_{1}^{E^{1,p_{0}-1}},\tilde{X}_{2}\right\} $
an adapted basis for $\left(S^{E^{1,p_{0}-1}}\right)_{c_{p_{0}}}$.
Keeping on blowing-up along the branch until the point $c_{p_{1}}$,
we get a succession of adapted basis. Now, at the point $c_{p_{1}},$
we have to prove that the vector field $X_{1}^{E^{1,p_{1}-1}}$ satisfies
the inequality
\begin{equation}
\nu_{c_{p_{1}}}\left(X_{1}^{E^{1,p_{1}-1}}\right)\leq\nu_{c_{p_{1}}}\left(\tilde{X}_{2}^{E^{p_{0},p_{1}-1}}\right).\label{ineg.2}
\end{equation}
Indeed, if the above inequality does not hold then there is no hope
to obtain a vector field $Y$ which satisfies both inequalities (\ref{ineg.1})
and (\ref{ineg.2}), which means optimal for both curves $\left(S^{E^{1,p_{0}-1}}\right)_{c_{p_{0}}}$
and $\left(S^{E^{1,p_{1}-1}}\right)_{c_{p_{1}}}$. However, we can
establish the lemma below adapted to a chain of curves of respective
types 
\[
\underset{\mathfrak{2}}{c_{p_{0}}}\leftarrow\underset{\mathfrak{1}}{c_{p_{0}+1}}\leftarrow\underset{\mathfrak{1}}{c_{p_{0}+2}}\leftarrow\cdots\leftarrow\underset{\mathfrak{1}}{c_{p_{1}-1}}\leftarrow\underset{\mathfrak{2}}{c_{p_{1}}}.
\]
 
\begin{lem}
There exists a vector field $Y$ optimal for $\left(S^{E^{1,p_{0}-1}}\right)_{c_{p_{0}}}$
such that $Y^{E^{p_{0},p_{1}-1}}$ is optimal for $\left(S^{E^{1,p_{1}-1}}\right)_{c_{p_{1}}}.$
\end{lem}

\begin{proof}
Applying inductively the property $\left(2\right)$ of Theorem \ref{thm.12}
to $\left(S^{E^{1,p_{1}-1}}\right)_{c_{p_{1}}}$ yields a germ of
smooth curve $l$ such that 
\[
\left(\left(S\cup l\right)^{E^{1,p_{1}-1}}\right)_{c_{p_{1}}}
\]
is not of type $\left(\mathfrak{E}_{d}^{\prime}\right)$. Since $\left(S^{E^{1,p_{0}-1}}\right)_{c_{p_{0}}}$
is of class $\mathfrak{2}$ then $\left(\left(S\cup l\right)^{E^{1,p_{0}-1}}\right)_{c_{p_{0}}}$
is of classe $\mathfrak{1}.$ Let us consider a an adapted basis $\left\{ Y_{1},Y_{2}\right\} $
for $\left(\left(S\cup l\right)^{E^{1,p_{0}-1}}\right)_{c_{p_{0}}}$.
Since the latter is of class $\mathfrak{1}$, one has 
\[
\nu_{c_{p_{0}}}\left(Y_{1}\right)=\nu_{c_{p_{0}}}\left(Y_{2}\right).
\]
The vector fields $Y_{1}$ and $Y_{2}$ leave invariant the smooth
curve $l^{E^{1,p_{0}-1}}.$ Thus there exists a germ of analytic function
$\phi$ such that the vector field 
\[
Y_{1}-\phi Y_{2}
\]
can be divided by a reduced equation $L$ of the curve $l^{E^{1,p_{0}-1}}.$
Therefore, according to the criterion of Saito \cite{MR586450}, the
curve $\left(S^{E^{p_{0}-1}}\right)_{c_{p_{0}}}$ admits a Saito basis
of the form 
\[
\left\{ \tilde{Y}_{1}=\frac{Y_{1}-\phi Y_{2}}{L},Y_{2}\right\} .
\]
Notice that $Y_{2}$ is still tangent to $l^{E^{p_{0}-1}}$ and that
$\nu_{c_{p_{0}}}\left(\tilde{Y}_{1}\right)<\nu_{c_{p_{0}}}\left(Y_{2}\right).$
In particular, $\tilde{Y}_{1}$ is optimal for $S^{E^{1,p_{0}-1}}.$
Now suppose that 
\[
\nu_{c_{p_{1}}}\left(\tilde{Y}_{1}^{E^{p_{0},p_{1}-1}}\right)\geq\nu_{c_{p_{1}}}\left(Y_{2}^{E^{p_{0},p_{1}-1}}\right)+1.
\]
Then, multiplying by $L^{E^{p_{0},p_{1}-1}}$ leads to
\begin{equation}
\nu_{c_{p_{1}}}\left(L^{E^{p_{0},p_{1}-1}}\tilde{Y}_{1}^{E^{p_{0},p_{1}-1}}\right)\geq\nu_{c_{p_{1}}}\left(Y_{2}^{E^{p_{0},p_{1}-1}}\right)+2\label{equation.Ed'}
\end{equation}
since $\nu_{c_{p_{1}}}\left(L^{E^{p_{0},p_{1}-1}}\right)=1.$ The
family 
\[
\left\{ L^{E^{p_{0},p_{1}-1}}\tilde{Y}_{1}^{E^{p_{0},p_{1}-1}},Y_{2}^{E^{p_{0},p_{1}-1}}\right\} 
\]
is a Saito basis for $\left(\left(S\cup l\right)^{E^{1,p_{1}-1}}\right)_{c_{p_{1}}}$.
However, the inequality (\ref{equation.Ed'}) implies that the latter
curve is of type $\left(\mathfrak{E}_{d}^{\prime}\right)$, which
is a contradiction with the choice of $l$. Therefore, one has 
\[
\nu_{c_{p_{1}}}\left(\tilde{Y}_{1}^{E^{p_{0},p_{1}-1}}\right)\leq\nu_{c_{p_{1}}}\left(Y_{2}^{E^{p_{0},p_{1}-1}}\right)
\]
and $\tilde{Y}_{1}$ satisfies the lemma.
\end{proof}
The property established in the lemma is also satisfied by $X_{1}^{E^{1,p_{1}-1}}$
since one can write 
\[
X_{1}^{E^{1,p_{1}-1}}=a\tilde{Y_{1}}+bY_{2}
\]
where $a$ is a unity. Thus, $X_{1}^{E^{1,p_{1}-1}}$ is optimal for
$\left(S^{E^{1,p_{1}-1}}\right)_{c_{p_{1}}}$ and, repeating the arguments
along the whole branch, we can see that the optimality property propagates.

Finally, for any branch $B$, we consider a vector field $X_{B}$
optimal along the branch $B$ and a generic combination of the form
\[
\sum\alpha_{B}X_{B},\ \alpha_{B}\in\mathbb{C}
\]
The latter is a Saito vector field for $S,$ which finishes the proof
of property $\left(1\right).$ 

Now, let us prove the second statement of Theorem \ref{thm.12}. Recall
that $X_{1}$ being Saito, its topological data provide an admissible
solution of the system $\left(\mathcal{H}\right)$. Since $S$ is
an union of smooth curves, for any $i=2,\ldots,M,$ 
\[
\delta_{i}^{S}=\Delta_{1}^{S}.
\]
In particular, the following relation holds

\begin{equation}
\epsilon_{1}^{S}=\frac{n_{1}}{2}+\underset{\nu_{1}}{\left\lfloor \begin{array}{c}
\Delta_{1}^{S}\\
\frac{1}{2}
\end{array}\right.}+\frac{M-1}{2}\Delta_{1}^{S}-\sum_{k=2}^{M}\underset{\nu_{k}-\Delta_{1}^{S}}{\left\lfloor \begin{array}{c}
\Delta_{k}^{S}\\
\frac{1}{2}
\end{array}\right.}.\label{super.equation}
\end{equation}
Suppose that $\nu\left(S_{1}\right)$ is even, then for any smooth
curve $l,$ the valuation $\nu\left(S_{1}\cup l\right)$ is odd, thus
$S\cup l$ cannot be of type $\left(\mathfrak{E}_{d}^{\prime}\right).$
Hence, we may suppose $\nu\left(S_{1}\right)$ odd. If $S$ is of
type $\left(\mathfrak{O}_{d}\right)$ then for any generic smooth
curve $S\cup l$ is of type $\left(\mathfrak{E}_{d}\right)$ and not
of type $\left(\mathfrak{E}_{d}^{\prime}\right).$ If $S$ is of type
$\left(\mathfrak{O}\right)$ then $S\cup l$ is of type $\left(\mathfrak{E}\right).$
Thus, we can also suppose that $S$ is of type $\left(\mathfrak{O}_{d}^{\prime}\right)$.
It remains a couple of cases to investigate
\begin{casenv}
\item $n_{1}^{S}>0.$ Let $l_{1}$ be some smooth component in $S$ attached
to $D_{1}$ and $l$ be germ of smooth curve tangent to $l_{1}$ at
order $1.$ We assert that $S\cup l$ cannot be of type $\left(\mathfrak{E}_{d}^{\prime}\right)$.
Indeed, if it was so, then the multiplicity of its Saito vector field
$X_{1}^{S\cup l}$ would be equal to
\begin{equation}
\nu\left(X_{1}^{S\cup l}\right)=\frac{\nu\left(S_{1}\right)+1}{2}-1=\frac{\nu\left(S_{1}\right)-1}{2}=\nu\left(X_{1}^{S}\right)\label{equation.n10}
\end{equation}
which is exactly the multiplicity of a Saito vector field for $S.$
However, one can obtain the topology of $X_{1}^{S\cup l}$ from the
one of $X_{1}^{S}$ provided that $S$ is of type $\left(\mathfrak{O}_{d}^{\prime}\right)$
and $S\cup l$ is of type $\left(\mathfrak{E}_{d}^{\prime}\right)$.
As depicted in Figure \ref{fig:Topology-of-the-1} it consists in
replacing the invariant smooth curve $l_{1}$ by two tangent smooth
curves $l$ and $l_{1}$ that are transverse after the first blowin-up.
\begin{figure}
\begin{centering}
\includegraphics[scale=0.3]{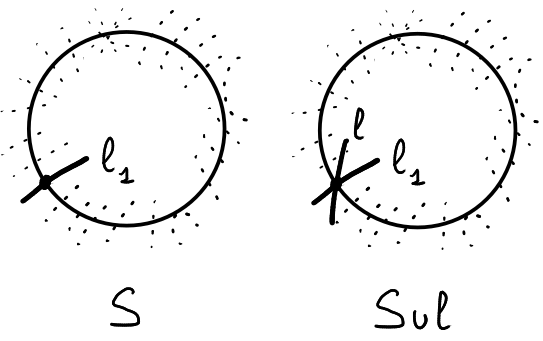}
\par\end{centering}
\caption{\label{fig:Topology-of-the-1}Topology of the Saito vector fields
of $S$ and $S\cup l$.}
\end{figure}
In the process, it can be seen that the valuation of the associated
vector field increases by one, which contradicts the equality (\ref{equation.n10}).
\item $n_{1}^{S}=0.$ Since $S$ is of type $\left(\mathfrak{O}_{d}^{\prime}\right)$,
it follows from \cite{genzmer2020saito} that one has 
\[
\epsilon_{1}^{S}=2-\sum_{k=2}^{M}\Delta_{k}^{S}.
\]
Combining with the relation \ref{super.equation} yields
\begin{equation}
\sum_{k=2}^{M}\underset{\nu_{k}\qquad\quad}{\left\lfloor \begin{array}{c}
0\\
\Delta_{k}^{S}-\frac{1}{2}
\end{array}\right.}=\frac{3}{2}.\label{equation.bien}
\end{equation}
Let us consider $l$ a germ of smooth curve such that $l^{E_{1}}$
is attached to $S_{2}$ and suppose that $S\cup l$ is of type $\left(\mathfrak{E}_{d}^{\prime}\right)$.
Applying the same arguments as above leads to 
\[
\sum_{k=3}^{M}\underset{\nu_{k}\qquad\quad}{\left\lfloor \begin{array}{c}
0\\
\Delta_{k}^{S\cup l}-\frac{1}{2}
\end{array}\right.}+\underset{\nu_{2}+1\qquad\quad}{\left\lfloor \begin{array}{c}
0\\
\Delta_{2}^{S\cup l}-\frac{1}{2}
\end{array}\right.}=2.
\]
However, for any $k\neq2,$ one has $\Delta_{k}^{S\cup l}=\Delta_{k}^{S}$
since $l$ is attached to $S_{2}$ and thus does not \emph{affect
}the curves $S_{k}$ for $k\neq2.$ Combining the two relations above
yields
\[
\underset{\nu_{2}+1\qquad\quad}{\left\lfloor \begin{array}{c}
0\\
\Delta_{2}^{S\cup l}-\frac{1}{2}
\end{array}\right.}-\frac{1}{2}=\underset{\nu_{2}\qquad\quad}{\left\lfloor \begin{array}{c}
0\\
\Delta_{2}^{S}-\frac{1}{2}
\end{array}\right.}.
\]
In particular, if $\nu_{2}$ is odd, then $\Delta_{2}^{S}=0.$ Thus,
if for any germ of smooth curve $l$ such that $l^{E_{1}}$ is attache
to $S_{i}$ for $i=2,\ldots,k$, the curve $S\cup l$ is of type $\left(\mathfrak{E}_{d}^{\prime}\right)$,
then we have the following alternative : either $\nu_{i}$ is even
or $\Delta_{i}^{S}=0$. But the latter contradicts (\ref{equation.bien}).
\end{casenv}
\end{proof}

\section{Number of moduli of $S.$}

Let us consider $\mathfrak{X}_{S}$ the sheaf of vector fields of
base $D_{1}=E_{1}^{-1}\left(0\right)$ whose stack $\left(\mathfrak{X}_{c}\right)_{p}$
is the set of germ of vector fields tangent to $E_{1}^{-1}\left(S\right).$
The cohomology of this sheaf is of finite dimension and we will denote
this dimension by $\sigma\left(S\right)$, 
\[
\sigma\left(S\right)=\dim_{\mathbb{C}}H^{1}\left(D_{1},\mathfrak{X}_{S}\right).
\]

Let us denote by $\mathfrak{T}^{k}\left(S\right)$ the tangency locus
between the strict transform $\left(E^{1,k}\right)^{\star}\left(S\right)$
of $S$ by $E^{1,k}$ and the exceptional divisor of the latter, 
\[
\mathfrak{T}^{k}\left(S\right)=\textup{Tan}\left(\left(E^{1,k}\right)^{\star}\left(S\right),\left(E^{1,k}\right)^{-1}\left(0\right)\right)
\]

According to \cite{YoyoBMS}, the number of moduli of $S$ - that
is the generic dimension of its moduli space - is equal to 
\begin{equation}
\textup{number of moduli of  \ensuremath{S}}=\sum_{k}\sum_{c\in\mathfrak{T}^{k}\left(S\right)}\sigma\left(\left(E^{1,k}\right)^{-1}\left(S\right)_{c}\right)\label{dim.moduli}
\end{equation}
where $\left(E^{1,k}\right)^{-1}\left(S\right)_{c}$ is the germ at
$c$ of the total transform of $S$ by $E^{1,k}.$ Moreover, following
\cite{genzmer2020saito}, the number $\sigma\left(S\right)$ can be
computed the following way : if $S$ is \emph{generic} in its moduli
space then

\begin{equation}
\sigma\left(C\right)=\left\{ \begin{array}{lr}
\text{\ensuremath{{\displaystyle \frac{\left(\nu\left(C\right)-2\right)\left(\nu\left(C\right)-4\right)}{4}}}} & \textup{if \ensuremath{C} is of type \ensuremath{\left(\mathfrak{E}\right)}}\\
{\displaystyle \ensuremath{\frac{\left(\nu\left(C\right)-3\right)^{2}}{4}}} & \textup{if \ensuremath{C} is of type \ensuremath{\left(\mathfrak{O}\right)}}\\
{\displaystyle \frac{\left(\nu\left(C\right)-2\right)\left(\nu\left(C\right)-4\right)}{4}-1+\epsilon_{1}^{C}+\sum_{k=2}^{M^{C}}\Delta_{k}^{C}} & \textup{if \ensuremath{C} is of type \ensuremath{\left(\mathfrak{E}_{d}\right)}}\\
{\displaystyle \frac{\left(\nu\left(C\right)-3\right)^{2}}{4}-2+\epsilon_{1}^{C}+\sum_{k=2}^{M^{C}}\Delta_{k}^{C}} & \textup{if \ensuremath{C} is of type \ensuremath{\left(\mathfrak{O}_{d}\right)}}
\end{array}\right.\label{dim.coho}
\end{equation}

Thus the expression of $\sigma\left(C\right)$ depends firstly, on
the type of the curve $C,$ secondly, on some topological data associated
to $C$ and its Saito foliation. When $C$ is an union of germs of
smooth curves, these data can be obtained from an admissible solution
of $\left(\mathcal{H}\right)$ since Proposition \ref{ExistenceAndUnicity}
and Theorem \ref{Saito.est.admiss} assert that this solution is unique
and given precisely by the topological data of a Saito foliation for
$C$. 

Therefore, to compute the number of moduli a curve $S$ which is an
union of smooth curves, one just has to solve the system $\left(\mathcal{H}\right)$
with an admissible solution for any curve $\left(E^{1,k}\right)^{-1}\left(S\right)_{c}$
and then to apply the formula (\ref{dim.coho}) to get the contribution
of each curve $\left(E^{1,k}\right)^{-1}\left(S\right)_{c}$ in the
expression (\ref{dim.moduli}).
\begin{example}
Let us consider the curve $S$ whose proximity matrix is given by
\[
P=\left(\begin{array}{ccc}
1 & -1 & 0\\
0 & 1 & -1\\
0 & 0 & 1
\end{array}\right)
\]
and $n_{1}^{S}=4,~n_{2}^{S}=2,\ n_{3}^{S}=4.$ Then we get the data
summarized in Table \ref{tab:Algorithm-to-compute}
\begin{table}
\begin{tabular}{cccc}
\noalign{\vskip\doublerulesep}
 & $S_{1}$ & $S_{2}\cup D_{1}$ & $S_{3}\cup D_{2}$\tabularnewline[\doublerulesep]
\noalign{\vskip\doublerulesep}
\hline 
\noalign{\vskip\doublerulesep}
Type & $\left(\mathfrak{E}_{d}\right)$ & $\left(\mathfrak{O}_{d}\right)$ & $\left(\mathfrak{O}_{d}\right)$\tabularnewline[\doublerulesep]
\noalign{\vskip\doublerulesep}
\hline 
\noalign{\vskip\doublerulesep}
Saito Picture & \includegraphics[scale=0.1]{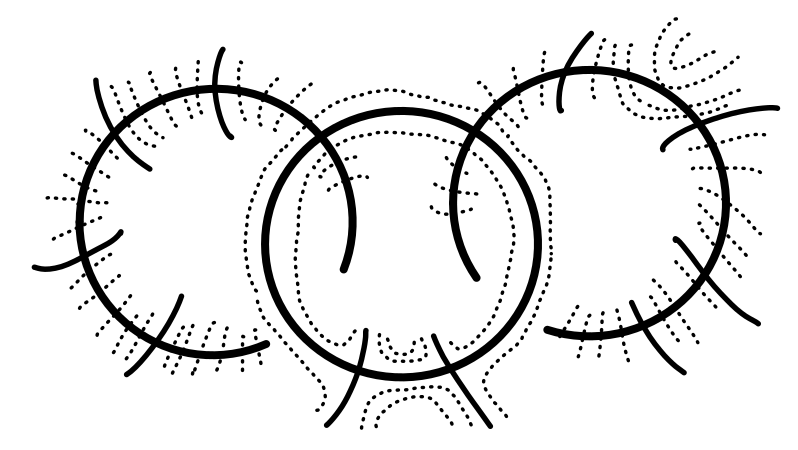} & \includegraphics[scale=0.1]{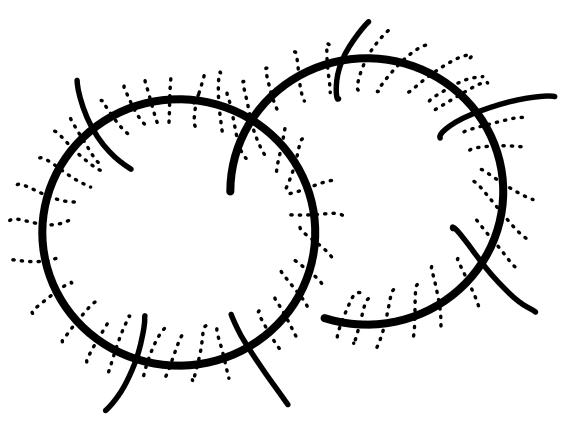} & \includegraphics[scale=0.1]{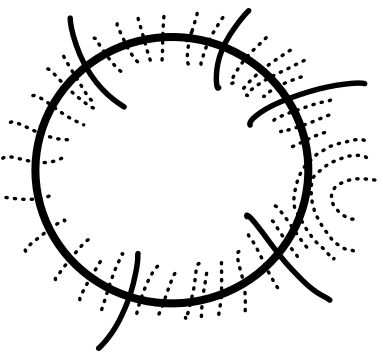}\tabularnewline[\doublerulesep]
\noalign{\vskip\doublerulesep}
\hline 
\noalign{\vskip\doublerulesep}
$P$ & $\left(\begin{array}{ccc}
1 & -1 & 0\\
0 & 1 & -1\\
0 & 0 & 1
\end{array}\right)$ & $\left(\begin{array}{cc}
1 & -1\\
0 & 1
\end{array}\right)$ & $\left(1\right)$\tabularnewline[\doublerulesep]
\noalign{\vskip\doublerulesep}
\hline 
\noalign{\vskip\doublerulesep}
$\nu$ & $10$ & $7$ & $5$\tabularnewline[\doublerulesep]
\noalign{\vskip\doublerulesep}
\hline 
\noalign{\vskip\doublerulesep}
$\nu\left(S_{k}\right)$ & $10,~6,\ 4$ & $7,\ 4$ & $5$\tabularnewline[\doublerulesep]
\noalign{\vskip\doublerulesep}
\hline 
\noalign{\vskip\doublerulesep}
$n_{i}^{S}$ & $4,~2,~4$ & $3,~4$ & $5$\tabularnewline[\doublerulesep]
\noalign{\vskip\doublerulesep}
\hline 
\noalign{\vskip\doublerulesep}
$\Delta_{i}^{S}$ & $\left(0,1,0\right)$ & $\left(0,0\right)$ & $\left(0\right)$\tabularnewline[\doublerulesep]
\noalign{\vskip\doublerulesep}
\hline 
\noalign{\vskip\doublerulesep}
$\delta_{i}^{S}$ & $\left(0,0,1\right)$ & $\left(0,0\right)$ & $\left(0\right)$\tabularnewline[\doublerulesep]
\noalign{\vskip\doublerulesep}
\hline 
\noalign{\vskip\doublerulesep}
$\mathfrak{S}^{S}$ & $\left(\begin{array}{c}
5\\
4\\
2
\end{array}\right)$ & $\left(\begin{array}{c}
4\\
2
\end{array}\right)$ & $\left(3\right)$\tabularnewline[\doublerulesep]
\noalign{\vskip\doublerulesep}
\hline 
\noalign{\vskip\doublerulesep}
$\mathcal{E}^{S}$ & $\left(\begin{array}{c}
1\\
2\\
2
\end{array}\right)$ & $\left(\begin{array}{c}
2\\
2
\end{array}\right)$ & $\left(3\right)$\tabularnewline[\doublerulesep]
\noalign{\vskip\doublerulesep}
\hline 
\noalign{\vskip\doublerulesep}
{\large{}$\sigma\left(S\right)$} & {\large{}13} & {\large{}4} & {\large{}2}\tabularnewline[\doublerulesep]
\noalign{\vskip\doublerulesep}
\end{tabular}

\caption{\label{tab:Algorithm-to-compute}Algorithm to compute the number of
moduli of $\Sigma.$}
\end{table}
\end{example}

Therefore, the number of moduli of $\Sigma$ is equal to 
\[
13+4+2=19.
\]

We implemented, among other procedures, this algorithm on Sage 9.{*}
. It can be found here
\begin{quote}
https://perso.math.univ-toulouse.fr/genzmer/
\end{quote}
It is a software called \emph{CourbePlane.}

\bibliographystyle{plain}
\bibliography{/home/genzmer/ownCloud/Article/Biblio/Bibliographie}

\end{document}